\documentclass[12pt,bezier]{article}
\usepackage{amssymb}
\usepackage{mathrsfs}
\usepackage{amsmath}
\usepackage{amsfonts,amsthm,amssymb}
\usepackage{amsfonts}
\usepackage{graphics}
\newtheorem{lem}{Lemma}
\newtheorem{thm}[lem]{Theorem}
\newtheorem{cor}[lem]{Corollary}

\begin{document}
\title{ On the regular 2-connected 2-path Hamiltonian graphs }
\author{Xia Li, \quad Weihua Yang\footnote{Corresponding author. E-mail: ywh222@163.com; yangweihua@tyut.edu.cn
(W.~Yang).} \\
\small Department of Mathematics,  Taiyuan University of
Technology,  Taiyuan 030024,  China\\}
\maketitle {\flushleft\bf Abstract:} {\small  A graph $G$ is $l$-path Hamiltonian if every path of length not exceeding $l$ is contained in a Hamiltonian cycle. It is well known that  a 2-connected, $k$-regular graph $G$ on at most $3k-1$ vertices is edge-Hamiltonian if for every edge $uv$ of $G$, $\{u,v\}$ is not a cut-set. Thus $G$ is 1-path Hamiltonian if $G\setminus \{u,v\}$ is connected for every edge $uv$ of $G$. Let $P=uvz$ be a 2-path of a 2-connected, $k$-regular graph $G$ on at most $2k$ vertices. In this paper, we show that there is a Hamiltonian cycle containing the 2-path $P$ if $G\setminus V(P)$ is connected. Therefore, the work implies a condition for  a 2-connected, $k$-regular graph to be 2-path Hamiltonian. An example shows that the $2k$ is almost sharp, i.e.,  the number is at most $2k+1$.}


\section{Introduction}
 All graphs mentioned in this paper are finite simple  graphs. Standard graph theory notation and terminology not explained in this paper, we refer the reader to~\cite{bondy}.
 A \emph{Hamiltonian cycle} in a graph $G$ is a cycle containing all the vertices of $G$, and a graph with a Hamiltonian cycle is called \emph{Hamiltonian}. Dirac's Theorem~\cite{Dirac} states that every $n$-vertex graph with minimum degree at least $\frac{n}2$ is Hamiltonian.

 One particular classic subarea on Hamiltonian graph theory is about Hamiltonian cycles containing specified elements of a graph. One of these directions is the study of $l$-path Hamiltonian. A graph $G$ on $n$ vertices is said to be \emph{$l$-path Hamiltonian} if every path of length not exceeding $l$, $1\leqslant l\leqslant n-2$, is contained in a Hamiltonian cycle (i.e., a Hamiltonian graph is 0-path Hamiltonian).  A graph $G$ is said to be \emph{edge-Hamiltonian, or 1-path Hamiltonian} if every edge of $G$ is contained in a Hamiltonian cycle. Kronk in~\cite{Hudson V. Kronk} considered the $l$-path Hamiltonian.

 \begin{thm}[\cite{Hudson V. Kronk}]\label{thm1}
Let $G$ be a graph on $n$ vertices, if $d(a)+d(b)\geqslant n+l$ for every pair of non-adjacent vertices $a$ and $b$, then $G$ is $l$-path Hamiltonian.
 \end{thm}

It is not difficult to see that Kronk's work is sharp. Due to the theorem above, we try to explore such problems on $k$-regular graphs.

Many problems and conjectures on Hamiltonian regular graphs have been investigated by various authors. The problem of determining the values of $k$ for which all 2-connected, $k$-regular graphs on $n$ vertices are Hamiltonian was first suggested by Szekeres (see \cite{B. Jackson}). Jackson in \cite{B. Jackson} showed that every 2-connected, $k$-regular graph on at most $3k$ vertices is Hamiltonian. The strongest result of these works given by Li in \cite{H. Li} is that all 2-connected, $k$-regular graphs, $k\geqslant 14$, on at most $3k+4$ vertices are Hamiltonian except two kinds of well defined families of graphs.

Li in~\cite{Hao Li} showed the following result that under almost the same conditions in \cite{B. Jackson}, the graphs are edge-Hamiltonian.

\begin{thm}[\cite{Hao Li}]\label{thm3}
Let $G$ be a 2-connected, $k$-regular graph on $n\leqslant 3k-1$ vertices, and let $e_{0}=uv$ be any edge of $G$ such that \{u, v\} is not a cut-set, then $G$ has a Hamiltonian cycle containing $e_{0}$.
\end{thm}

In other words, if $G$ is a 2-connected, $k$-regular graph on at most $3k-1$ vertices, and $G\setminus V(P)$ is connected for every path $P$ of length 1, then $G$ is 1-path Hamiltonian.

By Theorem~\ref{thm1}, we have that 2-connected, $k$-regular graphs on at most $2k-2$ vertices are 2-path Hamiltonian. Naturally, what else can we say about the 2-path Hamiltonian regular graphs?  In this paper, we are going to prove the following.

\begin{thm}\label{thm2}
Let $G$ be a 2-connected, $k$-regular graph on $n\leqslant 2k$ vertices, and let $P=uvz$ be any path of  $G$ such that $\{u, v, z\}$ is not a cut-set, then $G$ has a Hamiltonian cycle containing $P$.
\end{thm}

The following corollary  follows from Theorem~\ref{thm3} and Theorem~\ref{thm2}.

\begin{cor}\label{thm4}
Let $G$ be a 2-connected, $k$-regular graph on at most $2k$ vertices, if $G\setminus V(P)$ is connected for every path $P$ of length at most 2, then $G$ is 2-path Hamiltonian.
\end{cor}

We shall present an example which shows that the best bound of Theorem~\ref{thm4} is at most $2k+1$. Let $H_{i}$, $i=1,2$, be a graph which is obtained from $K_{k+1}$ by deleting one edge $e_{i}=a_{i}b_{i}$. We can construct a 2-connected, $k$-regular graph $G$ on 2k+2 vertices from two disjoint copies $H_{1}$ and $H_{2}$ by adding $a_{1}a_{2}$ and $b_{1}b_{2}$.  There is a 2-path in $G$ that is not contained in any Hamiltonian cycle of $G$. Thus, the problems on regular 2-connected $l$-path Hamiltonian graphs with $n$ vertices are interesting in $2k-l\leqslant n\leqslant 2k+1$.

\section{Proof of Theorem~\ref{thm2}}
The proof of Theorem~\ref{thm2} is divided into two cases. We first consider the case of $k\geqslant5$ and we prove it by using the classic hopping lemma (\cite{Woodall}, Lemma 12.3).  In the end, we consider the cases of $k=3$ and $k=4$.

 We fist assume $k\geqslant5$. Let $G$  be a 2-connected, $k$-regular graph on $n\leqslant 2k$ vertices, and let $P=uvz$ be a path of  $G$ such that $\{u, v, z\}$ is not a cut-set. We define a new graph $G_{1}$ by inserting two vertices $w_{1}$ and $w_{2}$ on the edges $e_{1}=uv$ and $e_{2}=vz$ of $P$ respectively. Then we have $G_{1}=(G-\{e_{1}, e_{2}\})\cup\{w_{1}, w_{2}\}\cup\{uw_{1}, w_{1}v, vw_{2}, w_{2}z\}$, $P_{1}=u w_{1} v w_{2} z$ and $|V(G_{1})|=n_{1} \leqslant 2k+2$. Clearly, it is sufficient to prove that $G_{1}$ is Hamiltonian.

Suppose that $G_{1}$ is not Hamiltonian. Let $C_{1}=c_{1}, c_{2},  \cdots , c_{n_{1}-r_{1}}$ be a longest cycle of $G_{1}$ containing $w_{1}$ and $w_{2}$ (Note that $G_{1}-V(P_{1})$ is connected.), such that the number of components of $R_{1}=G_{1}-C_{1}$ is as small as possible. Let $r_{1}=|R_{1}|$, $R_{1}'$ be the largest component of $R_{1}$ and $r_{1}'=|R_{1}'|$. The subscripts of $c_{i}$ will be reduced modulo $n_{1}-r_{1}$ throughout. Obviously, we have $|V(C_{1})|=n_{1}-r_{1} \geqslant 6$.

For any $A, B \subseteq V(G_{1})$, let

\begin{center}
  $e(A, B)=|\{uv\in E(G_{1}):u\in A, v\in B\}|$

  $e(A)=|\{uv\in E(G_{1}):u, v\in A\}|$.
\end{center}

For any $D \subseteq V(C_{1})$, let
\begin{center}
  $D^{+}=\{c_{i+1}:c_{i}\in D\}$ and $D^{-}=\{c_{i-1}:c_{i}\in D\}$.
\end{center}

$\bf Case \ 1$. $R_{1}$ contains an isolated vertex $v_{0}$.

Define that $Y_{0}=\emptyset$, and for any $j \geqslant 1$,
\begin{center}
$X_{j}=N(Y_{j-1}\cup\{v_{0}\})$

$Y_{j}=\{c_{i}\in C_{1}:c_{i-1}, c_{i+1}\in X_{j}\}$
\end{center}
and
\begin{center}
$X=\bigcup\limits_{i = 1}^\infty  {{X_j}},\  \ Y=\bigcup\limits_{i = 0}^\infty  {{Y_j}}, \ \ \   x=|X|\geqslant k \  \  \  and   \  \     \    y=|Y| . $
\end{center}
By the hopping lemma, we have $X \subset V(C_{1}),\  X \cap Y = \emptyset $ and X dose not contain two consecutive vertices of $C_{1}$.

Let $S_{1}, S_{2},  \cdots , S_{x}$ be the sets of vertices contained in the open segments of $C_{1}$ between vertices of $X$. Put $\phi = \{ {S_i}:\left| {{S_i}} \right| \geqslant2, 1\leqslant i\leqslant x\}.  $
Then $S_{i}=\{c_{l}, c_{l+1}, \cdots , c_{m}\}\in\phi $ is said to be $\psi $-connected to $S_{j}=\{c_{q}, c_{q+1}, \cdots , c_{z}\}\in\phi $ if $\left| {{S_i}} \right|$ is odd and $c_{q}$ and $c_{z}$ are both joined to $c_{l+e}$ for all odd $e$, $1\leqslant e \leqslant m-l-1$. Now, $c_{l+1}, c_{l+3}, \cdots, c_{m-1}$ are called $P$-vertices of $S_{i}$.
Set $P= $$ \{$$c_{i}\in V(C_{1})$ : $c_{i}$ is a $P$-vertex of some $S_{j}$ which is $\psi$-connected to some $S_{t}$ of $\phi$$\}$, and $p=|P|$.

Since
$$e(V(G_{1}) - X, X) = (n_{1} - 2 - x)k + 4 - 2e(V(G_{1}) - X)$$
\begin{equation}\label{eq1}
e(X, V(G_{1}) - X) \leqslant xk
\end{equation}
we have
\begin{equation}\label{eq2}
2e(V(G_{1}) - X) \geqslant (n_{1} - 2 - 2x)k + 4.
\end{equation}
On the other hand, under the properties of $C_{1}$, we can follow the series of the arguments in \cite{B. Jackson} and finally have the following inequality:
\begin{equation}\label{eq3}
\begin{split}
2e(V(G_{1}) - X) \leqslant p(k + n_{1} - x - y - p - 1) + (n_{1} - 2x - 2p - 1)(n_{1} - 2x)\\ - p(x - y - p) - 2(r_{1} - 1)(x - y - 1).\ \ \ \ \  \ \  \ \ \  \ \ \ \ \ \ \  \ \  \ \ \  \ \
\end{split}
\end{equation}
Combining (\ref{eq3}) with (\ref{eq2}), it can be deduced that
\begin{equation}\label{eq4}
p + 4 \leqslant (n_{1} - 2x - k)(n_{1} - 1 - 2x - p) + k - 2(r_{1} - 1)(x - y - 1).
\end{equation}
By the definitions of $X$ and $Y$, we have $x\geqslant y$. If $x=y$, we have
\[e(Y \cup \left\{ {{v_0}} \right\}, X) = k + (y - 2)k + 4 = xk + 4 - k, \]
contrary to (\ref{eq1}) because of $k \geqslant 5$. It follows that $2(r_{1} - 1)(x - y - 1)\geqslant 0$.
From the definition of $P$, we have  $p \leqslant \frac{n_{1} - 1 - 2x}2$, which implies $n_{1} - 1 - 2x -p\geqslant 2p - p \geqslant 0$. And $k \geqslant n_{1} - 2x - p - 1$ by $n_{1}\leqslant 2k+2$ and $x\geqslant k$. So we have
\begin{equation}\label{eq5}
p + 4 \leqslant (n_{1} - 2x - k + 1)k - 2(r_{1} - 1)(x - y - 1).
\end{equation}
Therefore by (\ref{eq5}), we have $n_{1} - 2x - k + 1 > 0$, and then $n_{1} > 3k - 1$, a contradiction.

The next two cases in this part are both discussed that $R_{1}$ contains no isolated vertex.

For a path $Q = {q_1}, {q_2},  \cdots , {q_g}, g \geqslant 2$, in $R_{1}$, let $t(Q)$ denote the number of occurrences of ordered pair $(c_{i}, c_{j})$ of the vertices of $C_{1}$ such that $c_{i}$ is joined to one of $q_{1}$ and $q_{g}$, $c_{j}$ is joined to the other, and $e(\left\{ {{q_1}, {q_g}} \right\}, \{{c_{i + 1}}, {c_{i + 2}},  \cdots , {c_{j - 1}}\}) = 0$. We say that $Q$ satisfies the condition $(*)$ if $t(Q)\geqslant 2$,  ${N_{C_{1}}}(\left\{ {{q_1}, {q_g}} \right\}) \not\subset
 \left\{ {u, v, z} \right\}$ and there is a ordered pair $(c_{i}, c_{j})$ of the vertices of $C_{1}$ such that $u, v, z, {w_1}\ and \ {w_2} \notin \left\{ {{c_{i + 1}}, {c_{i + 2}},  \cdots , {c_{j - 1}}} \right\}$. Put $A = {N_{C_{1}}}({q_1})$ and $B = {N_{C_{1}}}({q_g})$.

$\bf Case \  2. $ $2\leqslant|R_{1}'|\leqslant k-1$.

Before the proof of this case, we derive some results about the structure of $R_{1}'$.
\begin{lem}\label{lem1}
There exists a maximal path $Q$ in $R_{1}'$ such that $Q$ satisfies $(*)$.
\end{lem}

\begin{proof}
Since $|R_{1}'|\leqslant k-1$,  for any $v_{i}\in V(R_{1}'), \ i=1 , \cdots , r_{1}' $, we have ${N_{C_{1}}}({v_i}) \geqslant 2$. By the assumption of 2-connectivity and $\{u, v, z\}$ is not a cut-set, there exists a path $Q = {q_1}, {q_2},  \cdots , {q_g}$ in $R_{1}'$, which is chosen as long as possible such that $Q$ satisfies $(*)$.

If $Q$ is not a maximal path of $R_{1}'$, let $Q' ={b_1}, {b_2},  \cdots , {b_s}, {q_1}, {q_2},  \cdots , {q_g}, {q_{g+1}},  \cdots , {q_e}$ be a maximal path in $R_{1}'$ containing $Q$. Without loss of generality, we assume $s \geqslant 1$.

From the definition of $Q$, it is easy to see that ${N_{C_{1}}}({b_1}) \leqslant 3$. So ${N_{R_{1}^{'}}}({b_1}) \geqslant k-3$, and there is at most one vertex in $R_{1}^{'}$ which is not adjacent to ${b_1}$.

We consider the following two cases.

  Case (a):  ${b_1}{q_2}\in E(G_1)$.

In this case, there is a longer path $Q^{''} ={q_1}, {b_s}, {b_{s - 1}}, \cdots , {b_1}, {q_2},  \cdots , {q_g}$ than $Q$ that satisfies $(*)$, a contradiction of the definition of $Q$.

  Case (b):  ${b_1}{q_2} \notin  E(G_1)$.

In this case, if $s\geqslant 2$, there is a longer path $Q^{''} ={q_1}, {b_s}, {b_{s - 1}}, \cdots , {b_1}, {q_3},  \cdots , {q_g}$ than $Q$ that satisfies $(*)$. If $s=1$, we claim that ${N_{C_1}}({q_2})\leqslant 3 $, otherwise, $Q^{''} ={q_2}, {q_1},{b_1}, {q_3},  \cdots , {q_g}$ is a longer path than $Q$ that satisfies $(*)$. Therefore, ${q_2}$ is joined to every vertex of $R_{1}'$ except ${b_1}$. There is a longer path $Q^{''} ={q_1}, {b_1}, {q_3}, {q_2}, {q_4}, \cdots , {q_g}$ than $Q$ that satisfies $(*)$, a contradiction.

A similar argument holds if $e>g$.
\end{proof}

\begin{lem}\label{lem2}
There exists a maximal path $Q$ in $R_{1}'$ such that $t(Q)\geqslant 3$.
\end{lem}

\begin{proof}
Suppose that $Q$ satisfies the property of Lemma~\ref{lem1} and $t(Q)=2$. Then we consider the following two cases:

case (a): $A=B=\{c_{i}, c_{j}\} \not\subset \{u, v, z\}$, $c_{i}\neq v$ and $c_{j}\neq v$;

case (b): $A=\left\{ {{c_{{i_1}}}, {c_{{i_2}}},  \cdots , {c_{{i_s}}}} \right\}$ and $B=\left\{ {{c_{{j_1}}}, {c_{{j_2}}},  \cdots , {c_{{j_l}}}} \right\}$ such that $s\geqslant 2, l\geqslant 2$ and $\left\{ {{c_{{i_1}}}, {c_{{i_2}}},  \cdots , {c_{{i_s}}}} \right\} \cap \left\{ {{c_{{j_1}}}, {c_{{j_2}}},  \cdots , {c_{{j_l}}}} \right\} = \emptyset $.

If case (a) occurs, we have $g=r_{1}'=k-1$. Without loss of generality, let $\{w_{1}, w_{2}\}\notin \left\{ {{c_{i + 1}}, {c_{i + 2}},  \cdots , {c_{j - 1}}} \right\} $, then we have $c_{d}={{\rm{c}}_i}^ - $ or $c_{d'}={{\rm{c}}_j}^ + $ such that $c_{d}\notin \{{w_1, w_2}\}$ or $c_{d'}\notin \{{w_1, w_2}\}$. Clearly, we have
\[{N_{C_1}}({c_{d}}) \cap \left[ {Q \cup \left\{ {{c_{j - 1}}, {c_{j- 2}},  \cdots , {c_{j- g}}} \right\} \cup c_{d} } \right] = \emptyset \]
or
 \[{N_{C_1}}({c_{d'}}) \cap \left[ {Q \cup \left\{ {{c_{i + 1}}, {c_{i + 2}},  \cdots , {c_{i + g}}} \right\} \cup c_{d'} } \right] = \emptyset . \]

 And there is at least two of $\{v, w_{1}, w_{2}\}$ which can not be adjacent to $c_{d}$ or $c_{d'}$. It follows that \[{d_{C_1}}({c_{d}}) \leqslant 2k + 2 - 2(k - 1) - 3 = 1  \] or \[{d_{C_1}}({c_{d'}}) \leqslant 2k + 2 - 2(k - 1) - 3 = 1.\] This is a contradiction.

For case (b), without loss of generality, let $w_{1}, \ w_{2}\notin \{{c_{{j_1}}}, {c_{j_1 + {1}}}, \cdots, {c_{{j_l}}}\}$, then there exists either some ${c_z} \in {A^ + }$ satisfying  ${N_{C_1}}({c_z}) \cap \left[ {Q \cup (\bigcup\limits_{h = 1}^{l - 1} {\left\{ {{c_{j_{h} + 1}}, {c_{j_{h} + 2}}} \right\}} ) \cup \left\{ {{w_1}, {w_2}, {c_{z}}} \right\}} \right] = \emptyset $,  or some ${c_f} \in {A^ - }$  satisfying ${N_{C_1}}({c_f}) \cap \left[ {Q \cup (\bigcup\limits_{h = 2}^{l } {\left\{ {{c_{j_{h} - 1}}, {c_{j_{h} - 2}}} \right\}} ) \cup \left\{ {{w_1}, {w_2}, {c_{f}}} \right\}} \right] = \emptyset $. Which implies \[{d_{C_1}}({c_z}) \leqslant 2k + 2 - \left[ {g + 2(l - 1) + 3} \right] \leqslant k - 2\]
or\[ {d_{C_1}}({c_f}) \leqslant 2k + 2 - \left[ {g + 2(l - 1) + 3} \right] \leqslant k - 2\]

a contradiction.
\end{proof}

\begin{cor}\label{cor1}
If $t(Q)\geqslant 3$, then $g\leqslant k-2$.
\end{cor}

\begin{proof}
Suppose $g\geqslant k-1$ and $t(Q)\geqslant 3$. Then $|A\cup B|\geqslant3$, we have \[\begin{array}{l}
2k + 2 \geqslant \left| {V(G_{1})} \right| = \left| R_{1} \right| + \left| {V(C_{1})} \right| \\
 \ \ \ \ \ \ \ \ \ \geqslant r_{1} + \left| {A \cup B} \right| + (t(Q) - 2)g + 2  \ \ \ \ \\
 \ \ \ \ \ \ \ \ \       \geqslant 2(k - 1) + 3 + 2 = 2k + 3
\end{array}\]
a contradiction.
\end{proof}

\begin{lem}\label{lem3}
There exists a maximal path $Q$ in $R_{1}'$ such that $t(Q)\geqslant 3$. Then $A=B$.
\end{lem}

\begin{proof}
By contradiction. Suppose $B\neq A$ and $|B-A|\geqslant 1$, without loss of generality, $|B|\geqslant |A|$. We have \[\begin{array}{l}
V(C_{1}) = n_{1} - r_{1} \geqslant \left| {A \cup B} \right| + \left| {{A^ + } \cup {A^ - } \cup {B^ + } \cup {B^ - }} \right|{\rm{ + (t - 2)(g - 2)}}\\
\ \ \ \ \ \ \ \  = \left| A \right|{\rm{ + }}\left| {B{\rm{ - }}A} \right| + \left| {{A^ + } \cup {B^ - }} \right|{\rm{ + }}\left| {({A^ - } \cup {B^ + }) - ({A^ + } \cup {B^ - })} \right|{\rm{ + (t - 2)(g - 2)}}.
\end{array}\]

Since $g\geqslant2$, if $c_{d}\in {A^ + } \cap {B^ - }$, we have $c_{d}=w_{1}$ or $c_{d}=w_{2}$. Let $t(Q) = t$, $\sigma  = \left| {B{\rm{ - }}A} \right| + \left| {({A^ - } \cup {B^ + }) - ({A^ + } \cup {B^ - })} \right|$, and\[{ \theta{ = }}\left\{ \begin{array}{l}
 - 2  \ \ \ \ if\  u\in A ,  \  v\in A\cap B, \  z\in B\\
 - 1 \ \ \ \  if\ u\in A , \   v\in B\backslash A  \  \  or \   \  v\in A\backslash B, \   z\in B\\
0\ \ \ \ \ \ otherwise.
\end{array} \right. \]

So we have \[n_{1} - r_{1} \geqslant \left| A \right| + \left| {{A^ + }} \right| + \left| {{B^ - }} \right| + \theta  + \sigma  + (t - 2)(g - 2). \]

By the maximality of $Q$, $\left| A \right| \geqslant k - g + 1$ and $\left| B \right| \geqslant k - g + 1$. Therefore we have \[\begin{array}{l}
  2k + 2 - r_{1} \geqslant \left| A \right| + \left| {{A^ + }} \right| + \left| {{B^ - }} \right| + \theta  + \sigma  + (t - 2)(g - 2)\\
2k + 2 - r_{1} \geqslant 3(k - g + 1) + \theta  + \sigma  + (t - 2)(g - 2)\\
 \ \ \ \ \ \ \ \ 1 - \theta  \geqslant r_{1} - g + k - g + \sigma  + (t - 3)(g - 2).
\end{array}\]

By Corollary~\ref{cor1}, we have \[ - 1 - \theta  \geqslant r_{1} - g + \sigma  + (t - 3)(g - 2). \]

Since $t(Q)\geqslant 3$,  $g\geqslant 2$, $r_{1}\geqslant g$ and $\sigma \geqslant  1$, we have a contradiction when $\theta=0$ or $\theta=-1$. If $\theta=-2$, we have $({A^ - } \cup {B^ + }) - ({A^ + } \cup {B^ - }) = \emptyset $ which implies $\theta=0$, a contradiction. In fact, let $c_{i}\in B-A$ be the vertex such that the next vertex of $A \cup B$ after  $c_{i}$ belongs to A. Since $c_{i+1} \notin ({A^ - } \cup {B^ + }) - ({A^ + } \cup {B^ - })$, we have $c_{i+1}$ is in ${B^ - }$, which implies $c_{i+2}\in B\cap A$ and then ${c_i} = u \in B - A$ and ${c_{i + 2}} = v \in B \cap A$, or ${c_i} = v \in B - A$ and ${c_{i + 2}} = z \in B \cap A$. According to the definition of $\theta$, we see $\theta=0$.
\end{proof}

\begin{lem}\label{lem4}
There exists a maximal path $Q$ in $R_{1}'$ such that $t(Q)\geqslant 3$. Then $g=k-t+1$.
\end{lem}

\begin{proof}
Clearly, $g\geqslant k-t+1$. If $g\geqslant k-t+2$, by Lemma~\ref{lem3} and Corollary~\ref{cor1}, we have $2 \leqslant g \leqslant k - 2$. Thus, \[2k + 2 \geqslant V(C_{1}) + g \geqslant g(t - 2) + 2 + g + t \geqslant (g + 1)(t - 1) + 3 \geqslant (g + 1)(k - g + 1) + 3. \]

But since $f(g)=(g + 1)(k - g + 1) + 3$ is a concave function of $g$ and $f(2) = f(k - 2) = 3k > 2k + 2$, we have $f(g)> 2k+2$, a contradiction.
\end{proof}

Now, let $Q = {q_1}, {q_2},  \cdots , {q_g}$ be a maximal path in $R_{1}^{'}$ such that $t(Q)\geqslant 3$ and $A=B$. We  write ${X'} = A = B = \left\{ {x_1', x_2',  \cdots , x_t'} \right\}$.

Put $D{\rm{ = }}\left\{ {{S_i}, 1 \leqslant i \leqslant t} \right\}$, where ${{S_i}}$ is the set of vertices contained in the open segment of $C_{1}$ between two vertices of ${X'}$. Let $D'{\rm{ = }}\left\{ {{S_i^*}, i=1, 2} \right\}$ denote the element of $D$ which contains $w_{1}$ or $w_{2}$ (If $w_{1}$ and $w_{2}$ is contained in a same segment, let $D'=S^*$ ). Let $D^{''}=D-D'$.
The structure of $D$ has two cases:

Case (a): $w_{1}$ and $w_{2}$ is contained in a same segment $S^*$.

By Lemma~\ref{lem4}, we have
\[\begin{array}{l}
n_{1} \geqslant \left| {V(C_{1})} \right| + \left| {V(R_{1})} \right| \\
\ \ \ \ \geqslant g(t - 1) + 2 + (\left| {{S^*}} \right| - 2) + \sum\limits_{{S_i} \in {D^{''}}} {(\left| {{S_i}} \right| - g)}  + t + g + (r_{1} - g)\\
 \ \  \ \ \geqslant (g + 1)t + 2 + (\left| {{S^*}} \right| - 2) + \sum\limits_{{S_i} \in {D^{''}}} {(\left| {{S_i}} \right| - g)}  + (r_{1} - g)\\
 \ \ \ \ \geqslant (g + 1)(k - g + 1) + 2 + (\left| {{S^*}} \right| - 2) + \sum\limits_{{S_i} \in {D^{''}}} {(\left| {{S_i}} \right| - g)}  + (r_{1} - g).
\end{array}\]

Put $f'(g) = (g + 1)(k - g + 1) + 2$. Since $f'(g)$ is a concave function of $g$ with $f'(2) = 3k-1= f'(k - 2)$, we obtain a contradiction that \[2k + 2 \geqslant 3k - 1 + (\left| {{S^*}} \right| - 2) + \sum\limits_{{S_i} \in {D^{''}}} {(\left| {{S_i}} \right| - g)}  + (r_{1} - g). \]

Case (b):  $w_{1}$ is contained in $S_1^*$, and $w_{2}$ is contained in $S_2^*$.

By Lemma~\ref{lem4}, we have
\[\begin{array}{l}
n_{1} \geqslant \left| {V(C_{1})} \right| + \left| {V(R_{1})} \right| \\ \ \ \ \ \geqslant g(t - 2) + 2 + \sum\limits_{i = 1}^2 {(\left| {S_i^*} \right| - 1)}  + \sum\limits_{{S_i} \in {D^{''}}} {(\left| {{S_i}} \right| - g)}  + t + g + (r_{1} - g)\\
 \ \ \ \ \geqslant (g + 1)(t - 1) + 3 + \sum\limits_{i = 1}^2 {(\left| {S_i^*} \right| - 1)}  + \sum\limits_{{S_i} \in {D^{''}}} {(\left| {{S_i}} \right| - g)}  + (r_{1} - g)\\
 \ \ \ \ \geqslant (g + 1)(k - g) + 3 + \sum\limits_{i = 1}^2 {(\left| {S_i^*} \right| - 1)}  + \sum\limits_{{S_i} \in {D^{''}}} {(\left| {{S_i}} \right| - g)}  + (r_{1} - g).
\end{array}\]

Put $f''(g) = (g + 1)(k - g ) + 3$. When $2\leqslant g\leqslant k-3$, $f''(g)$ is a concave function of $g$ with $f''(2) = 3k-3= f''(k - 3)$, we have \[2k + 2 \geqslant 3k - 3 + \sum\limits_{i = 1}^2 {(\left| {S_i^*} \right| - 1)}  + \sum\limits_{{S_i} \in {D^{''}}} {(\left| {{S_i}} \right| - g)}  + (r_{1} - g). \]

 There is a contradiction when $k\geqslant 6$ from \[5-k \geqslant \sum\limits_{i = 1}^2 {(\left| {S_i^*} \right| - 1)}  + \sum\limits_{{S_i} \in {D^{''}}} {(\left| {{S_i}} \right| - g)}  + (r_{1} - g). \]
When $k=5$, we have $r_{1}=g$, $|S_i|=g$ for all $S_i \in D^{''}$ and $|S_i^*|=1$ for $i=1, 2 $. By Lemma~\ref{lem4}, we have $t=k-g+1=6-g$. For any elements $S_i$ and $S_j$ of $D^{''}$, we have $e(S_i, S_j)=0$ because of the maximality of $C_{1}$.
Firstly, if there is some $q_{i} \ of \ Q-\{q_{1}\} $ such that  ${N_{C_{1}}}(q_{i})\cap S_i \neq \emptyset$ for some $S_i \in D^{''}$. By Lemma~\ref{lem4}, since $q_{i-1}q_{g} \in E(G_{1})$, then $Q' ={q_1}, {q_2},  \cdots , {q_{i-1}}, {q_g}, {q_{g-1}},  \cdots , {q_i}$ is a  path satisfying $(*)$ in $R_{1}^{'}$, which implies $2g+1\leqslant g$. So we have ${N_{C_{1}}}(Q)\subset X'$. Secondly, we have $ e(X', V(G_{1})-X')\leqslant kt=5t$.
Moreover, we also have  \[kt \geqslant e(V(G_{1})-X', X')\geqslant gt+(t-2)g(k-g+1)+4. \]

By Lemma 6 we deduce that \[(5-g)(6-g) \geqslant (4-g)g(6-g)+4. \]

Because $ 2\leqslant g \leqslant 3$ and $g$ is an integer, we have \[(5-g)(6-g) \leqslant (4-g)g(6-g)+4\] a contradiction.

If $g=k-2$, we have $t=3$. So there exists $x_i' \in {X'}$ such that $x_i^{'-} \notin \{w_{1}, w_{2}\}$ or $x_i^{'+} \notin \{w_{1}, w_{2}\}$. It is clearly that \[d_{G_{1}}(x_i^{' - }) \leqslant 2k + 2 - 2(k - 2) - 2 = 4\]
or\[d_{G_{1}}(x_i^{' + }) \leqslant 2k + 2 - 2(k - 2) - 2 = 4\]
a contradiction.

$\bf Case \ 3. $ $|R_{1}'|\geqslant k. $

By the assumption of connectivity and $\{u, v, z\}$ is not a cut-set, there exists $x{''} \in {N_{C_{1}}}(R_{1}^{'})$, such that $x^{''-} \notin \{w_{1}, w_{2}\}$. It is clearly that ${N_{G_{1}}}({x^{''-}}) \cap {R'} = \emptyset $, and at least two of $\{v, w_{1}, w_{2}\}$ cannot be adjacent to $x^{''-}$. It follows that \[d_{G_{1}}(x^{''-}) \leqslant 2k + 2 -k - 2-1 = k-1\] a contradiction.

These contradictions complete our proof in this part. We next discuss the cases of $k=3$ and $k=4$. Similarly, let $C$ be a longest cycle of $G$ containing $P$ and $R=G-C$. Clearly, $|C|\geqslant4$.
By Theorem~\ref{thm1}, we only need to discuss the cases that $2k-1\leqslant|V(G)|=n\leqslant 2k$.

When $k=3$, $5\leqslant n\leqslant 6$. If $n=5$, we consider the following two cases.

Case (a): $|C|=5$. Theorem~\ref{thm2} holds.

Case (b): $|C|=4$. Let $C=u v z x_{1}$. Then $R$ is an isolated vertex $v_{0}$. It is easy to see that there exist two consecutive vertices of $\{u, z, x_{1}\}$ which are adjacent to $v_{0}$. A contradiction of that $C$ is the longest cycle of $G$ containing $P$.

If $n=6$, we consider the following three cases.

Case (a): $|C|=6$. Theorem~\ref{thm2} holds.


Case (b): $|C|=5$. Let $C=u v z x_{1} x_{2}$. Then $R$ is an isolated vertex $v_{0}$. By assumption, we have ${N_C}(v_{0})=\{v, z, x_{2}\}$ or ${N_C}(v_{0})=\{u, v, x_{1}\}$. By symmetry, we consider the case of ${N_C}(v_{0})=\{v, z, x_{2}\}$. Since $ux_{1}\in E(G)$, there is a Hamiltonian cycle $C'=u, v, z, v_{0}, x_{2}, x_{1}, u$ containing $P$. Theorem~\ref{thm2} holds.

Case (c): $|C|=4$. Let $C=u v z x_{1}$.

Subcase (c1): $R$ contains an isolated vertex $v_{0}$. It is similar to that of the case(b) when $k=3$ and $n=5$.

Subcase (c2): $R$ contains no isolated vertex. So the vertices of $C$ are adjacent to $R$. This contradict with the assumption that $C$ is the longest cycle of $G$ containing $P$.

When $k=4$, $7 \leqslant n \leqslant 8$. If $n=7$, we consider the following four cases.

Case (a): $|C|=7$. Theorem~\ref{thm2} holds.

Case (b): $|C|=6$. Let $C=u v z x_{1} x_{2} x_{3}$. Then $R$ is an isolated vertex $v_{0}$. By assumption, we have ${N_C}(v_{0})=\{u, v, z, x_{2}\}$. If $x_{1}x_{3}\notin E(G)$, we have $x_{1}u, x_{1}v\in E(G)$ which makes ${d_G}(x_{3})\leqslant3$, a contradiction. So we have $x_{1}x_{3}\in E(G)$. There is a Hamiltonian cycle $C'=u, v, z, v_{0}, x_{2}, x_{1}, x_{3}, u$ containing $P$. Theorem~\ref{thm2} holds.

Case (c): $|C|=5$. Let $C=u v z x_{1} x_{2}$.

Subcase (c1): $R$ contains an isolated vertex $v_{0}$. It is easy to
see that there exist two consecutive vertices of $\{u, z, x_{1}, x_{2}\}$ which are adjacent to $v_{0}$, a contradiction.

Subcase (c2): $R$ contains no isolated vertex. Since $G$ is 4-regular graph, we have ${d_C}(R)\geqslant 6$ and ${N_C}(R)\geqslant 3$. When ${N_C}(R)=3$, by assumption, we have ${N_C}(R)=\{v, z, x_{2}\}$ or ${N_C}(R)=\{u, v, x_{1}\}$. By symmetry, we consider the case of ${N_C}(R)=\{v, z, x_{2}\}$ in which we have ${d_G}(u)\leqslant3$, a contradiction. When ${N_C}(R)\geqslant 4$, there exist two consecutive vertices of $\{u, z, x_{1}, x_{2}\}$ which are adjacent to $R$, a contradiction.

Case (d): $|C|=4$. Let $C=u v z x_{1}$. For every connected component $R'$ of $R$, ${N_C}(R')\geqslant 3$. Clearly, there exist two consecutive vertices of $\{u, z, x_{1}\}$ which are adjacent to $R'$, a contradiction.

If $n=8$, we consider the following five cases.

Case (a): $|C|=8$. Theorem~\ref{thm2} holds.

Case (b): $|C|=7$. Let $C=u v z x_{1} x_{2} x_{3} x_{4}$. Then $R$ is an isolated vertex $v_{0}$. By assumption, we have ${N_C}(v_{0})=\{u, v, z, x_{2}\}$, ${N_C}(v_{0})=\{u, v, z, x_{3}\}$, ${N_C}(v_{0})=\{u, v, x_{1}, x_{3}\}$ or ${N_C}(v_{0})=\{v, z, x_{2}, x_{4}\}$. By the same discussion as for n=7 when k=4, there is a Hamiltonian cycle containing $P$ in all cases.

Case (c): $|C|=6$. Let $C=u v z x_{1} x_{2} x_{3}$.

Subcase (c1): $R$ contains two isolated vertices $v_{0}$ and $v_{1}$. By assumption, we have ${N_C}(v_{0})={N_C}(v_{1})=\{u, v, z, x_{2}\}$, ${d_G}(x_{1})\leqslant3$, a contradiction.

Subcase (c2): $R$ is an edge $e=v_{0}v_{1}$. Since $G$ is a 4-regular graph, we have  ${d_C}(R)\geqslant 6$ and ${N_C}(R)\geqslant 3$. When ${N_C}(R)=3$, we have ${N_C}(v_{0})={N_C}(v_{1})$. By assumption, we have ${N_C}(R)=\{u, v, x_{1}\}$, ${N_C}(R)=\{v, z, x_{3}\}$, ${N_C}(R)=\{u, v, x_{2}\}$,
${N_C}(R)=\{v, z, x_{2}\}$, ${N_C}(R)=\{u, z, x_{2}\}$ or ${N_C}(R)=\{x_{1}, x_{3}, v\}$. In the discussion of all cases, either there is a contradiction of regularity, or there is a Hamiltonian cycle containing $P$.
When ${N_C}(R)\geqslant 4$, it is clear that there is no consecutive vertices of $\{u, z, x_{1}, x_{2}, x_{3}\}$ which are adjacent to $R$. So ${N_C}(R)=\{u, v, z, x_{2}\}$.
We claim $x_{1}x_{3}\in E(G)$, and then there is a Hamiltonian cycle $C'=u, v, z, v_{1}, v_{0}, x_{2}, x_{1}, x_{3}, u$ containing $P$. Otherwise, $ux_{1}, vx_{1}\in E(G)$ which makes ${d_G}(x_{3})\leqslant3$, a contradiction.

Case (d): $|C|=5$. It is similar to that of the case(c) when $k=4$ and $n=7$.

Case (e): $|C|=4$. Let $C=u v z x_{1}$. Obviously, $R$ contains no isolated vertex.
Let $R'$ be a connected component of $R$. If ${N_C}(R')\geqslant 3$, it is clear that there exist two consecutive vertices of $\{u, z, x_{1}\}$ which are adjacent to $R'$, a contradiction. If ${N_C}(R')= 2$, we have ${N_C}(R)=\{v, x_{1}\}$, which makes ${d_G}(u)\leqslant3$, a contradiction.

Thus, we complete the proof.

\end{document}